\documentclass[12pt]{amsart}
\usepackage{amssymb}
\newtheorem{theorem}{Theorem}
\newtheorem{lemma}[theorem]{Lemma}
\newtheorem{proposition}[theorem]{Proposition}
\newtheorem{corollary}[theorem]{Corollary}
\newtheorem{example}[theorem]{Example}
\newtheorem{question}{Question}
\newtheorem{definition}[theorem]{Definition}

\title{Cardinal invariants for the $G_\delta$ topology}
\author{Angelo Bella}
\address{
Department of Mathematics and Computer Science \\
University of Catania \\
Citt\'a universitaria\\  
viale A. Doria 6 \\
95125 Catania, Italy}
\email{bella@dmi.unict.it}

\author{Santi Spadaro}

\address{
Department of Mathematics and Computer Science \\
University of Catania \\
Citt\'a universitaria\\  
viale A. Doria 6 \\
95125 Catania, Italy}
\email{santidspadaro@gmail.com}



\begin{document}

\maketitle

\begin{abstract}
We prove upper bounds for the spread, the Lindel\"of number and the weak Lindel\"of number of the $G_\delta$ topology on a topological space and apply a few of our bounds to give a short proof to a recent result of Juh\'asz and van Mill regarding the cardinality of a $\sigma$-countably tight homogeneous compactum.
\end{abstract}

\section{Introduction}

All spaces are assumed to be $T_1$. The word \emph{compactum} indicates a compact Hausdorff space.

Given a topological space $X$ we can consider a finer topology on $X$ by declaring countable intersections of open subsets of $X$ to be a base. The new space is called the \emph{$G_\delta$ topology of $X$} and is denoted with $X_\delta$.

There are various papers in the literature investigating what properties of $X$ are preserved when passing to $X_\delta$ and presenting bounds for cardinal invariants on $X_\delta$ in terms of the cardinal invariants of $X$ (see for example \cite{Juhasz}, \cite{FW}, \cite{Pytkeev}, \cite{KMS}). Moreover, results of that kind have found applications to central topics in general topology like the study of covering properties in box products (see, for example, \cite{Kunen}), cardinal invariants for homogeneous compacta (see, for example \cite{Arhangel'skiiGdelta}, \cite{Carlson}, \cite{CPR} and \cite{SS}) and spaces of continuous functions (See \cite{Arhangel'skiiCp}).

Two of the early results on this topic are Juh\'asz's bound $c(X_\delta) \leq 2^{c(X)}$ for every compact Hausdorff space $X$, where $c(X)$ denotes the \emph{cellularity} of $X$ and Arhangel'skii's result that the $G_\delta$ topology on a Lindel\"of regular scattered space is Lindel\"of. Juh\'asz's bound is tight in the sense that it's not possible to prove that $c(X_\delta) \leq c(X)^\omega$ for every compact space $X$ (see \cite{Fleissner}) and the scattered property is essential in Arhangel'skii's result because there are compact Hausdorff spaces whose $G_\delta$-topology even has (weak) Lindel\"of number $\mathfrak{c}^+$ (see \cite{SS}).

In this paper we prove various new bounds for cardinal invariants on the $G_\delta$ topology. For example we prove that $s(X_\delta) \leq 2^{s(X)}$ for every space $X$, where $s(X)$ is the spread of $X$ (that is, the supremum of the cardinalities of the discrete subsets of $X$). For a regular space we prove that $L(X_\delta) \leq \min{ \{psw(X)^{d(X)}, 2^{s(X)} \}}$, where $L(X)$ denotes the Lindel\"of degree of $X$, $psw(X)$ denotes the point-separating weight of $X$ and $d(X)$ denotes the density of $X$.

Many questions are left open. For example we don't know whether the inequality $t(X_\delta) \leq 2^{t(X)}$ is true, where $t$ denotes the \emph{tightness}, even when $X$ is a compact space.

Finally, we exploit a few of our results to give a short proof of a recent result of Juh\'asz and van Mill on the cardinality of homogeneous compacta.

Our notation regarding cardinal functions follows \cite{Ju}. The remaining undefined notions can be found in \cite{E}. 

In our proofs we often use elementary submodels of the structure $(H(\mu), \epsilon)$. Dow's survey \cite{D} is enough to read our paper, and we give a brief informal refresher here. Recall that $H(\mu)$ is the set of all sets whose transitive closure has cardinality smaller than $\mu$. When $\mu$ is regular uncountable, $H(\mu)$ is known to satisfy all axioms of set theory, except the power set axiom. We say, informally, that a formula is satisfied by a set $S$ if it is true when all existential quantifiers are restricted to $S$. A set $M \subset H(\mu)$ is said to be an elementary submodel of $H(\mu)$ (and we write $M \prec H(\mu)$) if a formula with parameters in $M$ is satisfied by $H(\mu)$ if and only if it is satisfied by $M$. 

The downward L\"owenheim-Skolem theorem guarantees that for every $S \subset H(\mu)$, there is an elementary submodel $M \prec H(\mu)$ such that $|M| \leq |S| \cdot \omega$ and $S \subset M$. This theorem is sufficient for many applications, but it is often useful (especially in cardinal bounds for topological spaces) to have the following closure property. We say that $M$ is \emph{$\kappa$-closed} if for every $S \subset M$ such that $|S| \leq \kappa$ we have $S \in M$. For large enough regular $\mu$ and for every countable set $S \subset H(\mu)$ there is always a $\kappa$-closed elementary submodel $M \prec H(\mu)$ such that $|M|=2^{\kappa}$ and $S \subset M$.

The following theorem is also used often: let $M \prec H(\mu)$ such that $\kappa + 1 \subset M$ and $S \in M$ be such that $|S| \leq \kappa$. Then $S \subset M$.

\section{Cardinal invariants for the $G_\delta$ topology}

Let's start by listing the simplest bounds for cardinal functions of the $G_\kappa$ topology. They are probably folklore, and we include them just for the convenience of the reader.

\begin{proposition} \label{propeasy}{ \ \\}
\begin{enumerate}
\item $w(X_\kappa) \leq (w(X))^\kappa$.
\item $\chi(X_\kappa) \leq (\chi(X))^\kappa$.
\item If $X$ is regular then $d(X_\kappa) \leq 2^{d(X) \cdot \kappa}$.
\item If $X$ is regular, then $\pi w(X_\kappa) \leq 2^{\pi w(X) \cdot \kappa}$.
\end{enumerate}
\end{proposition}

\begin{proof}
The first two items are easy. 

As for the third item, recalling that $w(X) \leq 2^{d(X)}$ for regular spaces, we have that $d(X_\kappa) \leq w(X_\kappa) \leq w(X)^\kappa \leq 2^{d(X) \cdot \kappa}$.

To prove the fourth item, recall that $w(X) \leq (\pi w(X))^{c(X)}$ for every regular space $X$. Hence $\pi w(X_\kappa) \leq w(X_\kappa) \leq w(X)^\kappa \leq (\pi w(X))^{c(X) \cdot \kappa} \leq (\pi w(X))^{\pi w(X) \cdot \kappa} \leq 2^{\pi w(X) \cdot \kappa}$. 

\end{proof}

Regularity is essential in both the third and the fourth item, as the following example shows.

\begin{example}
A Hausdorff space $X$ such that: $$\pi w(X_\delta) \geq d(X_\delta)>2^{\pi w(X)} \geq 2^{d(X)}$$
\end{example}

\begin{proof}
Let $X=\beta \omega$, provided with the following topology: every principal ultrafilter is isolated. A basic neighbourhood of a non-principal ultrafilter $p$ has the form $\{p\} \cup A \setminus F$, where $A \in p$ and $F$ is a finite set. The space $X$ has a countable $\pi$-base, but $X_\delta$ is a discrete set of cardinality $2^{\mathfrak{c}}$.
\end{proof}

The following example shows that, unlike in the case of the $\pi$-weight, there is no bound on the $\pi$-character of the $G_\delta$-topology on a regular space of countable $\pi$-character.

\begin{example}
For every cardinal $\kappa$, there is a hereditarily normal space of countable $\pi$-character $X(\kappa)$ such that $\pi \chi(X(\kappa)_\delta) \geq \kappa$.
\end{example}

\begin{proof}
Let $X(\kappa)$ be the space obtained by taking the sum of a convergent sequence and the one-point compactification of a discrete set of size $\kappa$ and then collapsing the limit points to a single point $\infty$. In the resulting space, every point is isolated except for $\infty$, which nevertheless has a countable $\pi$-base. So $\pi \chi(X(\kappa))=\omega$. However, $X(\kappa)_\delta$ is homeomorphic to the one-point Lindel\"ofication of a discrete set of size $\kappa$. So its $\pi$-character is no smaller than $\kappa$.
\end{proof}

One of the early results regarding cardinal invariants for the $G_\delta$ topology was proved by Juh\'asz in \cite{Juhasz} and was originally motivated by a problem of Arhangel'skii regarding the weak Lindel\"of number of the $G_\delta$ topology on a compactum. Its proof is an application of the Erd\"os-Rado theorem from infinite combinatorics.

\begin{theorem}
(Juh\'asz) Let $X$ be a countably compact regular space. Then $c(X_\delta) \leq 2^{c(X)}$.
\end{theorem}

Note that regularity is essential in the above theorem as Vaughan \cite{V} constructed a countably compact Hausdorff space with points $G_\delta$ and cardinality larger than the continuum which is even separable.

We also exploit the Erd\"os-Rado theorem in our next result. Recall that $s(X)=\sup \{|D|: D$ is a discrete subset of $X \}$.

\begin{theorem} \label{spread}
Let $X$ be any space and $\kappa$ be a cardinal. Then $s(X_\kappa) \leq 2^{s(X) \cdot \kappa}$.
\end{theorem}

\begin{proof}
Without loss of generality we can assume that $s(X) \leq \kappa$. Suppose by contradiction that there is a discrete set $D \subset X_\kappa$ of cardinality $\geq (2^\kappa)^+$. For every $x \in D$ we can find a $G_\kappa$ set $G_x$ in $X$ such that $G_x \cap D=\{x\}$. Let $\{U^x_\alpha: \alpha < \kappa \}$ be a sequence of open sets such that $G_x=\bigcap \{U^x_\alpha: \alpha < \kappa \}$. Let $\prec$ be a linear ordering on $X$. For every $\alpha, \beta <\kappa$ let $C_{\alpha, \beta}=\{\{x,y\} \in [D]^2: x \prec y \wedge x \notin U^y_\alpha \wedge y \notin U^x_\beta \}$. Then $\{C_{\alpha, \beta}: (\alpha, \beta) \in \kappa^2 \}$ is a coloring of $[D]^2$ into $\kappa$ many colors. By the Erd\"os-Rado theorem we can find a set $T \subset D$ of cardinality $\kappa^+$ and a pair of ordinals $(\gamma, \delta) \in \kappa^2$ such that $[T]^2 \subset C_{\gamma, \delta}$. Note now that $U^x_\gamma \cap U^x_\delta \cap T=\{x\}$ for every $x \in T$. Hence $T$ is a discrete subset of $X$ of cardinality $\kappa^+$, which contradicts $s(X)=\kappa$.
\end{proof}

\begin{corollary}
(Hajnal and Juh\'asz) Let $X$ be a $T_1$ space. Then $|X| \leq 2^{s(X) \cdot \psi(X)}$.
\end{corollary}

\begin{proof}
Set $\kappa=s(X) \cdot \psi(X)$. By the above theorem we have $s(X_\kappa) \leq 2^\kappa$, but since $X_\kappa$ is discrete we must have $|X| \leq 2^\kappa$.
\end{proof}

The next example shows that $2^{s(X) \cdot \kappa}$ cannot be replaced with $s(X)^\kappa$ in Theorem $\ref{spread}$, even for compact LOTS.

\begin{example}
There is a compact linearly ordered space $L$ such that $s(L_\delta) > s(L)^\omega$.
\end{example}

\begin{proof}
Fleissner constructed in \cite{Fleissner} a compact linearly ordered space $L$ such that $c(L) \leq \mathfrak{c}$ and $L$ has a $\mathfrak{c}^+$-sized subset $S$ consisting of $G_\delta$ points. Since $c(X)=s(X)$ for every linearly ordered space $X$ we must have $s(L) \leq \mathfrak{c}$, but it's clear that $s(L_\delta) \geq \mathfrak{c}^+$.
\end{proof}

Recall that the Lindel\"of degree of a topological space $X$ ($L(X)$) is defined as the minimum cardinal $\kappa$ such that for every open cover of $X$ has a $\kappa$-sized subcover.

The weak Lindel\"of degree of $X$ ($wL(X)$) is defined as the minimum cardinal $\kappa$ such that, for every open cover $\mathcal{U}$ of $X$ there is a $\kappa$-sized subcollection $\mathcal{V} \subset \mathcal{U}$ such that $X \subset \overline{\bigcup \mathcal{V}}$.

At the 1970 International Congress of Mathematicians in Nice, France, Arhangel'skii asked whether the weak Lindel\"of degree of a compact space with its $G_\delta$ topology is always bounded by the continuum. A counterexample has recently been given in \cite{SS} but various related bounds for the (weak) Lindel\"of number of the $G_\delta$ topology have been presented in the literature (see, for example \cite{FW}, \cite{Pytkeev}, \cite{Juhasz} and \cite{CPR}).

A set $G \subset X$ is called a $G^c_\kappa$ set if there is a family $\{U_\alpha: \alpha < \kappa \}$ of open subsets of $X$ such that $G=\bigcap \{U_\alpha: \alpha < \kappa\}=\bigcap \{\overline{U_\alpha}: \alpha < \kappa \}$.

Given a space $X$, we denote with $X^c_\kappa$ the topology generated by the $G^c_\kappa$ subsets of $X$. Obviously if $X$ is regular, then $X_\kappa=X^c_\kappa$.

\begin{theorem} \label{mainthm}
Let $X$ be any space and $\kappa$ be a cardinal.Then $L(X^c_\kappa) \leq 2^{s(X) \cdot \kappa}$.
\end{theorem}

\begin{proof}
Without loss we can assume $s(X) \leq \kappa$. Fix a cover $\mathcal{F}$ of $X$ by $G^c_\kappa$ sets.

Let $\theta$ be a large enough regular cardinal and $M$ be a $\kappa$-closed elementary submodel of $H(\theta)$ such that $X, \mathcal{F} \in M$, $2^\kappa+1\subset M$ and $|M|=2^\kappa$.

For every $F \in \mathcal{F}$ choose open sets $\{U_\alpha(F): \alpha < \kappa\}$ witnessing that $F$ is a $G^c_\kappa$-set. Note that when $F \in \mathcal{F} \cap M$ we can assume that $\{U_\alpha(F): \alpha < \kappa \} \in M$ and hence $\{U_\alpha(F): \alpha < \kappa \} \subset M$.

\medskip

\noindent {\bf Claim 1.} $\mathcal{F} \cap M$ covers $\overline{X \cap M}$. 

\begin{proof}[Proof of Claim 1] 
Suppose this is not true and let $p \in \overline{X \cap M} \setminus \bigcup (\mathcal{F} \cap M)$. For every $x \in X \cap M$ we can find $F_x \in \mathcal{F} \cap M$ such that $x \in F_x$. Moreover, there must be $\alpha(x) <\kappa$ such that $p \notin U_{\alpha(x)}(F_x)$. Now, $\mathcal{O}=\{U_{\alpha(x)}(F_x): x \in X \cap M\}$ is an open cover of $X \cap M$. By Shapirovskii's Lemma (see \cite{Ju}) there is a discrete set $D \subset X \cap M$ and a subcollection $\mathcal{U} \subset \mathcal{O}$ with $|\mathcal{U}|=|D| \leq \kappa$ such that $X \cap M \subset \overline{D} \cup \bigcup \mathcal{U}$. By $\kappa$ closedness of $M$ we have $D, \mathcal{U} \in M$ hence $M \models X \subset \overline{D} \cup \bigcup \mathcal{U}$. Therefore by elementarity $H(\theta) \models X \subset \overline{D} \cup \bigcup \mathcal{U}$. Since $p \notin \bigcup \mathcal{U}$ we must have $p \in \overline{D}$.

Let now $F$ be an element of $\mathcal{F}$ such that $p \in F$. We have $p \in \overline{U_\alpha(F) \cap D}$ for every $\alpha < \kappa$ and $\overline{U_\alpha(F) \cap D} \in M$, by $\kappa$-closedness of $M$. Define $B=\bigcap \{\overline{U_\alpha(F) \cap D}: \alpha < \kappa \}$. Then $B \in M$. Note that we have $H(\theta) \models (\exists G \in \mathcal{F})(B \subset G)$, hence by elementarity $M \models (\exists G \in \mathcal{F})(B \subset G)$, which implies the existence of $H \in \mathcal{F} \cap M$ such that $p \in B \subset H$. But this contradicts the fact that $p \notin \bigcup (\mathcal{F} \cap M)$. Hence $\mathcal{F} \cap M$ covers $\overline{X \cap M}$ and the claim is proved.
\renewcommand{\qedsymbol}{$\triangle$}

\end{proof}

\noindent {\bf Claim 2.} $\mathcal{F} \cap M$ covers $X$.

\begin{proof}[Proof of Claim 2] Suppose this is not true and let $p$ be a point of $X \setminus \bigcup (\mathcal{F} \cap M)$. For every $F \in \mathcal{F} \cap M$ we can find $\beta(F) <\kappa$ such that $p \notin U_{\beta(F)}(F)$. 

It follows from Claim 1 that the family $\mathcal{V}:=\{U_{\beta(F)}(F): F \in \mathcal{F} \cap M\}$ is an open cover of $\overline{X \cap M}$. By Shapirovskii's Lemma we can find a discrete $D \subset \overline{X \cap M}$ and a family $\mathcal{W} \subset \mathcal{V}$ such that $|\mathcal{W}|=|D|<\kappa$ and $X \cap M \subset \overline{X \cap M} \subset \overline{D} \cup \bigcup \mathcal{W}$. Note that $D, \mathcal{W} \in M$, by $<\kappa$-closedness of $M$. This implies that $M \models X \subset \overline{D} \cup \bigcup \mathcal{W}$ and hence $H(\theta) \models X \subset \overline{D} \cup \bigcup \mathcal{W}$ by elementarity. But this is a contradiction because $p \notin W$, for every $W \in \mathcal{W}$ and since $\overline{D} \subset \overline{X \cap M}$ we also have that $p \notin \overline{D}$.
\renewcommand{\qedsymbol}{$\triangle$}
\end{proof}

Since $|M| \leq 2^\kappa$ it follows that $\mathcal{F} \cap M$ is a $2^\kappa$-sized subfamily of $\mathcal{F}$ covering $X$ and hence we are done.

\end{proof}

It's not possible to replace $X^c_\delta$ with $X_\delta$ in the above result, as the following example shows.

\begin{example}
There are $T_1$ spaces $X$ of countable spread where $L(X_\delta)$ can be arbitrarily large.
\end{example}

\begin{proof}
Let $\kappa$ be a cardinal of uncountable cofinality and $\mu=cf(\kappa)$.  Define a topology on $X=\kappa$ by declaring sets of the form $[0, \alpha] \setminus F$ to be a base, where $\alpha$ is an ordinal less than $\kappa$ and $F$ is a finite set. It is easy to see that $s(X)=\omega$. Moreover $\{[0, \alpha]: \alpha < \kappa \}$ is an open cover of $X$ without subcovers of cardinality less than $\mu$ and hence $L(X_\delta)\geq \mu$.
\end{proof}

However, for regular spaces, the $G_\delta$ modification and the $G^c_\delta$ modification coincide, so we obtain the following result:

\begin{theorem}
Let $X$ be a regular space. Then $L(X_\kappa) \leq 2^{s(X) \cdot \kappa}$.
\end{theorem}

Recall that the tightness of a point $x$ in the space $X$ ($t(x,X)$) is defined as the minimum cardinal $\kappa$ such that for every subset $A$ of $X$ with $x \in \overline{A}\setminus A$ there is a subset $B \subset A$ such that $|B| \leq \kappa$ and $x \in \overline{B}$. The tightness of the space $X$ is then defined as $t(X)=\sup \{t(x,X): x \in X\}$. A space of countable tightness is also called \emph{countably tight}.

\begin{theorem} \label{ccharthm}
Let $X$ be a countably compact space with a dense set of points of countable character. Then $wL(X^c_\kappa) \leq 2^{t(X) \cdot wL_c(X) \cdot \kappa}$.
\end{theorem}

\begin{proof}
Without loss of generality we can assume that $wL_c(X) \cdot t(X) \leq \kappa$. Fix a cover $\mathcal{F}$ of $X$ by $G^c_\kappa$ sets

Let $\theta$ be a large enough regular cardinal and $M$ be a $\kappa$-closed elementary submodel of $H(\theta)$ such that $X, \mathcal{F} \in M$ and $|M|=2^\kappa$.

For every $F \in \mathcal{F}$ choose open sets $\{U_\alpha: \alpha < \kappa \}$ witnessing that $F$ is a $G^c_\kappa$ set.

\noindent {\bf Claim 1.} $\mathcal{F} \cap M$ covers $\overline{X \cap M}$. 

\begin{proof}[Proof of Claim 1] 
Let $x \in \overline{X \cap M}$ and use $t(X) \leq \kappa$ to fix a $\kappa$-sized set $A \subset X \cap M$ such that $x \in \overline{A}$. Note $A \in M$. Let $F \in \mathcal{F}$ be such that $x \in F$ and let $\{U_\alpha: \alpha < \kappa \}$ be a sequence of open sets witnessing that $F$ is a $G^c_\kappa$ set. 

Note that the set $B=\bigcap \{\overline{A \cap U_\alpha}: \alpha < \kappa \}$ is in $M$ and $x \in B \subset F$. Now $H(\theta) \models (\exists F \in \mathcal{F})(B \subset F)$. Hence $M \models (\exists F \in \mathcal{F})(B \subset F)$. Therefore we can find $G \in \mathcal{F} \cap M$ such that $x \in B \subset G$, which is what we wanted.
\renewcommand{\qedsymbol}{$\triangle$}
\end{proof}

\noindent {\bf Claim 2.} $\mathcal{F} \cap M$ has dense union in $X$.

\begin{proof}[Proof of Claim 2]
Suppose not and let $p \in X \setminus \overline{\bigcup \mathcal{F} \cap M}$ be a point of countable character. Fix a local base $\{V_n: n < \omega \}$ at $p$. 

For every $x \in \overline{X \cap M}$ pick $F_x \in \mathcal{F} \cap M$ such that $x \in F_x$ and let $\{V^x_\alpha: \alpha < \kappa \} \in M$ be a sequence of open sets witnessing that $F_x$ is a $G^c_\kappa$ set. Since $p \notin F_x$, there must be $\alpha < \kappa$ such that $p \notin \overline{V^x_\alpha}$. Hence there must be $n_x < \omega$ such that $V_{n_x} \cap V^x_\alpha = \emptyset$. let $U_n=\bigcup \{V^x_\alpha: n_x=n \}$. Then $\{U_n : n < \omega \}$ is a countable open cover of the countably compact space $\overline{X \cap M}$. So there is $k<\omega$ such that $\{U_n: n < k\}$ covers $\overline{X \cap M}$. Let now $\mathcal{U}=\{U^x_\alpha: n_x <k \}$. Then $\mathcal{U}$ covers $\overline{X \cap M}$, hence $wL_c(X) \leq \kappa$ implies the existence of $\mathcal{V} \in [\mathcal{U}]^\kappa$ such that $\overline{X \cap M} \subset \overline{\bigcup \mathcal{V}}$. But that implies $M \models X \subset \overline{\bigcup \mathcal{V}}$ and hence $H(\theta) \models X \subset \overline{\bigcup \mathcal{V}}$, which contradicts $V_k \cap (\bigcup \mathcal{V})=\emptyset$.
\renewcommand{\qedsymbol}{$\triangle$}
\end{proof}

\end{proof}

\begin{corollary}
(Alas, \cite{Al}) Let $X$ be a countably compact $T_2$ space with a dense set of points of countable character. Then $|X| \leq 2^{\psi_c(X) t(X) wL_c(X)}$.
\end{corollary}

\begin{corollary}
Let $X$ be a regular countably compact space with a dense set of points of countable character. Then $wL(X_\kappa) \leq 2^{wL_c(X) \cdot t(X) \cdot \kappa}$.
\end{corollary}

\begin{corollary}
Let $X$ be a normal countably compact space with a dense set of points of countable character. Then $wL(X_\kappa) \leq 2^{wL(X) \cdot t(X) \cdot \kappa}$.
\end{corollary}

In a similar way we can prove the following theorem:

\begin{theorem} \label{isthm}
Let $X$ be a space with a dense set of isolated points. Then $wL(X^c_\kappa) \leq 2^{wL_c(X) \cdot t(X) \cdot \kappa}$.
\end{theorem}

\begin{question}
Is it true that $wL(X^c_\kappa) \leq 2^{wL_c(X) \cdot t(X) \cdot \kappa}$ for any Hausdorff space $X$?
\end{question}

We call a cover $\mathcal{U}$ of a space $X$, strongly point-separating if $\bigcap \{\overline{U}: U \in \mathcal{U} \wedge x \in U\}=\{x\}$. 

We define $psw_s(X)$ to be the least cardinal $\kappa$ such that $X$ admits a strongly point-separating open cover of order $\kappa$. Obviously $psw_s(X)=psw(X)$ for every regular space $X$.

\begin{theorem}
Let $X$ be a $T_2$ space. Then $L(X_\kappa) \leq psw_s(X)^{L(X) \cdot \kappa}$.
\end{theorem}

\begin{proof}
Let $\lambda=psw_s(X)$ and fix a strongly point-separating open cover $\mathcal{U}$ of $X$ of order $\leq \lambda$. We can assume $L(X) \leq \kappa$. Let $\mathcal{F}$ be a $G_\kappa$ cover of $X$. Since $L(X) \leq \kappa$ we can assume that $\mathcal{F}$ is made up of $\kappa$-sized intersections of elements of $\mathcal{U}$. Let $M$ be a $\kappa$-closed elementary submodel of $H(\theta)$ such that $\lambda^\kappa \subset M$, $X, \mathcal{U}, \mathcal{F} \in M$ and $|M|=\lambda^\kappa$.

{\noindent \bf Claim 1.} $\mathcal{F} \cap M$ covers $\overline{X \cap M}$.

\begin{proof}[Proof of Claim 1]
Let $p \in \overline{X \cap M}$. Let $F \in \mathcal{F}$ be such that $p \in F$. Let $\{U_\alpha: \alpha <\kappa \} \subset \mathcal{U}$ be a family of open sets such that $F=\bigcap \{U_\alpha: \alpha < \kappa \}$. Let $x_\alpha$ be any point in $U_\alpha \cap M$. Note that for every $\alpha <\kappa$ we have that $\{U \in \mathcal{U} : x_\alpha \in U\}$ is an element of $M$ of cardinality $\lambda$. Therefore $\{U \in \mathcal{U}: x_\alpha \in U \} \subset M$ and hence $U_\alpha \in M$, for every $\alpha < \kappa$. By $\kappa$-closedness of $M$ we have $F=\bigcap \{U_\alpha: \alpha < \kappa \} \in M$, as we wanted.
\renewcommand{\qedsymbol}{$\triangle$}
\end{proof}

{\noindent \bf Claim 2.} $\mathcal{F} \cap M$ actually covers $X$.

\begin{proof}[Proof of Claim 2]
Suppose that is not true and let $p \in X \setminus \bigcup (\mathcal{F} \cap M)$. For every $x \in \overline{X \cap M}$, let $F_x \in \mathcal{F} \cap M$ such that $x \in F_x$ and let $\{U^x_\alpha: \alpha < \kappa \} \in M$ be a sequence of open sets such that $\bigcap \{U^x_\alpha: \alpha < \kappa \} =F_x$. We again have that $\{U^x_\alpha : \alpha < \kappa \} \subset M$ and hence, for every $x \in \overline{X \cap M}$ we can find an open neighbourhood $U_x \in M$ of $x$ such that $p \notin U_x$. The family $\mathcal{V}:=\{U_x: x \in \overline{X \cap M}\}$ is an open cover of the space $\overline{X \cap M}$, which has Lindel\"of number at most $\kappa$ and hence we can find $\mathcal{C} \in [\mathcal{V}]^\kappa$ such that $X \cap M \subset \overline{X \cap M} \subset \bigcup \mathcal{C}$. By $\kappa$-closedness of $M$ we have $\mathcal{C} \in M$ and hence the previous formula implies $M \models X \subset \bigcup \mathcal{C}$. By elementarity we get that $H(\theta) \models X \subset \bigcup \mathcal{C}$, which contradicts the fact that $p \notin \bigcup \mathcal{C}$.
\renewcommand{\qedsymbol}{$\triangle$}
\end{proof}
\end{proof}

\begin{corollary}
Let $X$ be a regular space. Then $L(X_\kappa) \leq psw(X)^{L(X) \cdot \kappa}$.
\end{corollary}

\begin{question}
Is $t(X_\delta) \leq 2^{t(X)}$ true for every (compact) $T_2$ space $X$?
\end{question}

\section{An application to homogeneous compacta}

\begin{definition}
Let $X$ be a topological space. A set $S \subset X$ is called \emph{subseparable} if there is a countable set $C \subset X$ such that $S \subset \overline{C}$.
\end{definition}

Since $w(X) \leq 2^{d(X)}$ for every regular space $X$ and the weight is hereditary every subseparable subspace of a regular topological space has weight at most continuum.

\begin{lemma} \label{lemJV}
(Juh\'asz and van Mill, \cite{JM}) Let $X$ be a $\sigma$-countably tight homogeneous compactum. Then $X$ contains a non-empty subseparable $G_\delta$-subset and has a point of countable $\pi$-character.
\end{lemma}

\begin{corollary}
Every $\sigma$-countably tight homogeneous compactum has character at most continuum.
\end{corollary}

\begin{proof}
Let $x \in X$ be any point. By homogeneity we can find a subseparable $G_\delta$ set $G$ containing $x$. Then $w(G) \leq 2^{\omega}$. So we can fix a continuum-sized family $\mathcal{U}$ of open neighbourhoods of $x$ such that $G \cap \bigcap \mathcal{U}_x=\{x\}$. Let $\{U_n: n < \omega \}$ be a countable family of open sets such that $G=\bigcap \{U_n: n < \omega \}$. Then $\mathcal{V}=\{U_n: n<\omega \} \cup \mathcal{U}$ is a continuum sized family of open subsets of $X$ such that $\bigcap \mathcal{V}=\{x\}$. Since $X$ is compact, this implies that $\chi(x,X) \leq 2^\omega$.
\end{proof}

\begin{theorem} \label{lemdense}
Let $X$ be a homogeneous compactum which is the union of countably many countably tight dense subspaces. Then $L(X_\delta) \leq 2^{\omega}$.
\end{theorem}

\begin{proof}
Let $\{X_n: n < \omega \}$, be a countable family of countably tight subspaces covering $X$. Let $\mathcal{U}$ be a $G_\delta$-cover of $X$. Without loss we can assume that for every $U \in \mathcal{U}$ there are open sets $\{O_n(U): n <\omega \}$ such that $\overline{O_{n+1}(U)} \subset O_n(U)$, for every $n <\omega$ and $U=\bigcap \{O_n(U): n < \omega \}$.

Let $\theta$ be a large enough regular cardinal and let $M$ be an $\omega$-closed elementary submodel of $H(\theta)$ such that $|M|=2^{\omega}$ and $M$ contains everything we need.

\noindent {\bf Claim}. $\mathcal{U} \cap M$ covers $\overline{X \cap M}$. 

\begin{proof}[Proof of Claim]

Let $x \in \overline{X \cap M}$. 

We claim that $x \in \overline{X_n \cap M}$, for every $n<\omega$. Indeed, fix $n<\omega$ and let $V$ be a neighbourhood of $x$. Pick $y \in V \cap (X \cap M)$. Then $y$ has a local base $\mathcal{U}_y \in M$ having cardinality continuum. By the assumptions on $M$, $\mathcal{U}_y \subset M$. Since $X_n$ is dense in $X$, $M$ reflects this and therefore for every $U \in \tau \cap M$ we have $U \cap X_n \cap M \neq \emptyset$. Hence for every $U \in \mathcal{U}_y$ we have $U \cap X_n \cap M \neq \emptyset$. It turns out that $V \cap X_n \cap M \neq \emptyset$, for every open neighbourhood $V$ of $x$, as we wanted.

Let $k<\omega$ be such that $x \in X_k$. Using the fact that $X_k$ has countable tightness we can choose a countable set $C_k \subset X_k \cap M$ such that $x \in \overline{C_k}$. Note that, since $M$ is countably closed, every subset of $C$ is an element of $M$. Since $\mathcal{U}$ covers $X$ there is $U \in \mathcal{U}$ such that $x \in U$. Note that $x \in \overline{O_i(U) \cap C}$, for every $i<\omega$. Let $B=\bigcap \{\overline{O_i(U) \cap C}: i < \omega \}$ and note that $B \in M$. We have $H(\theta) \models (\exists U \in \mathcal{U})(B \subset U)$. Since every free variable in the previous formula belongs to $M$, by elementarity we have $M \models (\exists U \in \mathcal{U})(B \subset U)$ and hence there is $U \in \mathcal{U} \cap M$ such that $x \in B \subset U$, which finishes the proof of the Claim.
\renewcommand{\qedsymbol}{$\triangle$}
\end{proof}

Let us now prove that $\mathcal{U} \cap M$ actually covers $X$, which will finish the proof.

Suppose this is not the case and let $p \in X \setminus \bigcup (\mathcal{U} \cap M)$. By the claim, for every $x \in \overline{X \cap M}$ we can pick a $U_x \in \mathcal{U} \cap M$ containing $x$. Then we can choose $m<\omega$ such that $p \notin O_m(U_x) \in M$. This means that we can cover $\overline{X \cap M}$ by an open family $\mathcal{V} \subset M$ such that $p \notin \bigcup \mathcal{V}$. By compactness we can then take a finite subfamily $\mathcal{F}$ of $\mathcal{U}$ such that $X \cap M \subset \bigcup \mathcal{F}$. Since $\mathcal{F} \in M$ this is equivalent to $M \models X \subset \bigcup \mathcal{F}$, which implies, by elementarity, $H(\theta) \models X \subset \bigcup \mathcal{F}$, and that is a contradiction because $p \in H(\theta) \setminus \bigcup \mathcal{F}$.

\end{proof}

\begin{lemma}
Let $X$ be a compact homogeneous space which is the union of finitely many countably tight subspaces. Then $L(X_\delta) \leq 2^{\omega}$.
\end{lemma}

\begin{proof}
Let $\mathcal{F}$ be a finite cover of $X$ by countably tight subspaces. We can find a non-empty open subset $V$ of $X$ such that $V \cap F$ is dense in $V$, whenever $V \cap F \neq \emptyset$ and $F \in \mathcal{F}$. Applying the argument proving Lemma $\ref{lemdense}$ to $\overline{V}$ we obtain that $L(\overline{V}_\delta) \leq 2^{\aleph_0}$. Using the homogeneity of $X$ we can find an open cover $\mathcal{V}$ of $X$ such that $L(\overline{V}_\delta) \leq 2^{\aleph_0}$, for every $V \in \mathcal{V}$. Choosing a finite subcover of $\mathcal{V}$ we see that $L(X_\delta) \leq 2^\omega$:
\end{proof}

The following lemma was noted independently by de la Vega and Ridderbos (see \cite{R} for the proof of a much more general statement).

\begin{lemma}
Let $X$ be a homogeneous space. Then $|X| \leq d(X)^{\pi \chi(X)}$.
\end{lemma}

\begin{theorem} (Juh\'asz and van Mill)
Let $X$ be a compact homogeneous space which is the union of countably many dense countably tight subspaces or of finitely many countably tight subspaces. Then $|X| \leq 2^\omega$. 
\end{theorem}

\begin{proof}
Use homogeneity to fix, for every $x \in X$, a subseparable $G_\delta$ set $G_x$ containing $x$. We have $w(G_x) \leq 2^\omega$. Note that $\mathcal{U}=\{G_x: x \in X \}$ is a $G_\delta$ cover of $X$, so there is $C \in [X]^{2^{\omega}}$ such that $X \subset \bigcup \{G_x: x \in C \}$. For every $x \in C$, we can fix a continuum-sized $D_x \subset G_x$, dense in $G_x$. Then $D=\bigcup \{D_x: x \in C \}$ is a dense subset of $X$ having cardinality at most continuum, proving that $d(X) \leq 2^\omega$. Using the above lemmas we obtain that $|X| \leq 2^{\omega}$.
\end{proof}


\begin{thebibliography}{10}
\bibitem{Arhangel'skiiCp} A.V. Arhangel'skii, \emph{Topological function spaces}, Kluwer Academic Publishers, Mathematics and its Applications, vol. 78, Dordrecht, Boston, London, 1992.
\bibitem{Arhangel'skiiGdelta} A.V. Arhangel'skii, \emph{$G_\delta$-modification of compacta and cardinal invariants}, Commentationes Mathematicae Universitatis Carolinae \textbf{47} (2006), 95--101.
\bibitem{Al} O.T. Alas, \emph{More topological cardinal inequalities}, Colloquium Mathematicae \textbf{65} (1993), pp. 165--168.
\bibitem{AMR} A. Arhangel'skii, J. van Mill and G.J. Ridderbos \emph{A new bound on the cardinality of power-homogeneous compacta}, Houston Journal of Mathematics \textbf{33} (2007), 781--793.
\bibitem{Bella} A. Bella, \emph{On two cardinal inequalities involving free sequences}, Topology and its Applications \textbf{159} (2012), 3640--3643.
\bibitem{Carlson} N. Carlson, \emph{The weak Lindel\"of degree and homogeneity}, Topology and its Applications \textbf{160} (2013), 508--512.
\bibitem{CPR} N. Carlson, J. Porter and G.J. Ridderbos, \emph{On cardinality bounds for homogeneous spaces and the $G_\kappa$-modification of a space}, Topology and its Applications \textbf{159} (2012), 311--332.
\bibitem{Delavega} R. de la Vega, \emph{A new bound on the cardinality of homogeneous compacta}, Topology Appl. \textbf{153} (2006), 2118-2123.
\bibitem{D} A. Dow, \emph{An introduction to applications of elementary submodels to topology}, Topology Proceedings \textbf{13} (1988), no. 1, 17--72. 
\bibitem{E} R. Engelking, \emph{General Topology}, PWN, Warsaw, 1977.
\bibitem{Fleissner} W. G. Fleissner, \emph{Some spaces related to topological inequalities proven by the Erd\"os-Rado Theorem}, Proceedings of the American Mathematical Society \textbf{71} (1978), 313--320.
\bibitem{FW} W. Fleischmann and S. Williams \emph{The $G_\delta$-topology on compact spaces}, Fundamenta Mathematicae \textbf{83} (1974), 143--149. 
\bibitem{Gewand} M.E. Gewand, \emph{The Lindel\"of degree of scattered spaces and their products}, Journal of the Australian Mathematical Society (series A), \textbf{37} (1984), 98--105.
\bibitem{Juhasz} I. Juh\'asz, \emph{On two problems of A.V.\ Archangel'skii}, General Topology and its Applications \textbf{2} (1972) 151-156. 
\bibitem{Ju} I. Juh\'asz, \emph{Cardinal Function in Topology - Ten Years Later}, Mathematical Centre Tracts, 123, Mathematisch Centrum, Amsterdam, 1980.
\bibitem{JM} I. Juh\'asz and J. van Mill, \emph{On $\sigma$-countably tight spaces}, preprint, arXiv:1607.00517.
\bibitem{KMS} M. Kojman, D. Milovich and S. Spadaro, \emph{Noetherian type in topological products}, Israel Journal of Mathematics \textbf{202} (2014), 195--225.
\bibitem{Kunen} K. Kunen, \emph{Paracompactness of box products of compact spaces}, Transactions of the American Mathematical Society \textbf{240} (1978), 307--316.
\bibitem{LR} R. Levy and M.D. Rice, \emph{Normal $P$-spaces and the $G_\delta$-topology}, Colloquium Mathematicum \textbf{44} (1981), 227--240.
\bibitem{Pytkeev} E.G. Pytkeev, \emph{About the $G_\lambda $-topology and the power of some families of subsets on compacta}, Colloq. Math. Soc. Janos Bolyai, 41. Topology and Applications, Eger (Hungary), 1983, pp.517-522.
\bibitem{R} G.J. Ridderbos, \emph{On the cardinality of power-homogeneous Hausdorff spaces}, Fundamenta Mathematicae \textbf{192} (2006), pp. 255--266.
\bibitem{SS} S. Spadaro and P. Szeptycki, \emph{$G_\delta$ covers of compact spaces}, preprint, arXiv:1605.05630.
\bibitem{S} S. Spadaro, \emph{Infinite games and chain conditions}, Fundamenta Mathematicae \textbf{234} (2016), 229--239.
\bibitem{V} J. Vaughan, \emph{Countably compact locally countable $T_2$ spaces}, Proceedings of the American Mathematical Society \textbf{80} (1980), 147--153.
\end{thebibliography}
\end{document}